\documentclass[9pt,conference,onecolumn]{IEEEtran}
\usepackage[utf8]{inputenc}
\usepackage{bbm}
\usepackage{defs}
\usepackage{algpseudocode}

\newcommand{\longversion}[2]{#1} 

\usepackage{graphicx}

\ifCLASSINFOpdf
\else
\fi
\begin{document}
%
\title{\huge Stein COnsistent Risk Estimator (SCORE) for hard thresholding}

\author{\IEEEauthorblockN{Charles-Alban Deledalle}
\IEEEauthorblockA{Institut de Mathématiques de Bordeaux\\
CNRS-Université Bordeaux 1\\
Bordeaux, France\\
Email: cdeledal@math.u-bordeaux1.fr}
\and
\IEEEauthorblockN{Gabriel Peyré}
\IEEEauthorblockA{CEREMADE\\
CNRS-Université Paris Dauphine\\
Paris, France\\
Email: gabriel.peyre@ceremade.dauphine.fr}
\and
\IEEEauthorblockN{Jalal Fadili}
\IEEEauthorblockA{GREYC\\
CNRS-ENSICAEN\\
Caen, France\\
Email: jalal.fadili@greyc.ensicaen.fr}}


%


\maketitle

\begin{abstract}
  In this work, we construct a risk estimator for
  hard thresholding which can be used as a basis
  to solve the difficult task of automatically selecting the threshold. 
  As hard thresholding is not even continuous,
  Stein's lemma cannot be used to get an unbiased estimator of degrees of freedom, hence of the risk.
  We prove that under a mild condition, our estimator of the degrees of freedom, although biased, is consistent. 
  Numerical evidence shows that our estimator outperforms another biased risk estimator proposed in \cite{jansen2011icv}.
\end{abstract}
\vspace{-0.1cm}


%
\IEEEpeerreviewmaketitle

\section{Introduction}
We observe a realisation $y \in \RR^P$ of the normal random vector $Y = x_0 + W$, $W \sim \Nn(x_0, \sigma^2\Id_P)$.
Given an estimator $y \mapsto x(y, \lambda)$ of $x_0$ evaluated at $y$ and parameterized by $\lambda$,
the associated Degree Of Freedom (DOF)
is defined as \cite{efron1986biased}
\begin{align}
  &\DOF{x}(x_0, \lambda)
  \triangleq
  \sum_{i=1}^P \frac{\cov(Y_i, x(Y_i, \lambda))}{\sigma^2}~.
\end{align}
The DOF plays an important role in model/parameter selection.
For instance, define the criterion 
\begin{align}\label{eq:erisk}
\small{
  \norm{Y - x(Y, \lambda))}^2
  \!-\! P \sigma^2
  \!+\! 2 \sigma^2 \EDOF{x}(Y, \lambda) ~.
}
\end{align}
\longversion{In the rest, we denote $\diverg$ the divergence operator.}{}
If $x(\cdot, \lambda)$ is weakly differentiable w.r.t. its first argument with an essentially bounded gradient, Stein's lemma 
\cite{stein1981estimation} implies that $\EDOF{x}(Y, \lambda)=\diverg\pa{x(Y, \lambda)}$ and \eqref{eq:erisk} (the SURE in this case) are 
respectively unbiased estimates of $\DOF{x}(x_0, \lambda)$ and of the risk $\EE_W \norm{x(Y, \lambda) - x_0}^2$.
In practice, \eqref{eq:erisk} relies solely on the realisation $y$ which
is useful for selecting $\lambda$ minimizing \eqref{eq:erisk}.

In this paper, we focus on Hard Thresholding (HT)
\begin{align}\label{eq:ht}
  y \mapsto \HT(y, \lambda)_i =
  \choice{
    0 & \text{if } |y_i| < \lambda~,\\
    y_i & \text{otherwise} ~.
  }
\end{align}
HT is is not even continuous, and the Stein's lemma does not apply, so that $\DOF{x}(x_0,\lambda)$ and the risk cannot be unbiasedly estimated \cite{jansen2011icv}.
To overcome this difficulty, we build an estimator that, although biased, turns out to enjoy good asymptotic properties. 
In turn, this allows efficient selection of the threshold $\lambda$.

\section{Stein COnsistent Risk Estimator (SCORE)}

\longversion{
Remark that the HT can be written as
\begin{align*}
  \HT(y, \lambda) &=
  \ST(y, \lambda)
  +
  D(y, \lambda)
  \\
  \text{where} \quad
  \ST(y, \lambda)_i &=
  \choice{
    y_i + \lambda & \text{if} \quad y_i < -\lambda\\
    0 &\text{if} \quad -\lambda \leq y_i < +\lambda\\
    y_i - \lambda & \text{otherwise}\\
  }\\
  \text{and} \quad
  D(y, \lambda)_i &=
  \choice{
    - \lambda & \text{if} \quad y_i < -\lambda\\
    0 &\text{if} \quad -\lambda \leq y_i < +\lambda\\
    + \lambda & \text{otherwise}\\
  },
\end{align*}
where $y \mapsto \ST(y, \lambda)$ is the soft thresholding operator.
Soft thresholding is a Lipschitz continuous function of $y$ with an essentially bounded gradient, and therefore, appealing to Stein's lemma,
an unbiased estimator of its DOF is given by $\EDOF{\ST}(Y, \lambda) = \diverg \ST(Y, \lambda)$. This DOF estimate at a realization $y$ is known
to be equal to $\#\{ |y| > \lambda \}$, i.e., the number of entries
of $|y|$ greater than $\lambda$
(see \cite{donoho1995adapting,zou2007degrees}).
The mapping $y \mapsto D(y, \lambda)$ is piece-wise constant with discontinuities
at $\pm \lambda$ so that Stein's lemma does not apply to estimate the DOF of hard thresholding.
To circumvent this difficulty, we instead propose an estimator of the DOF
of a smoothed version replacing
$D(\cdot, \lambda)$ by $\Gg_h \star D(., \lambda)$ where $\Gg_h$ is a Gaussian kernel
of bandwidth $h > 0$ and $\star$ is the convolution operator.
In this case $\Gg_h \star D(., \lambda)$ is obviously $C^\infty$ whose DOF
can be unbiasedly estimated as $\diverg \pa{\Gg_h \star D(., \lambda)(Y)}$.
To reduce bias (this will be made clear from the proof), we have furthermore introduced
a multiplicative constant, $\sqrt{\sigma^2\!+\!h^2}/\sigma$,
leading to the following DOF formula
}{
We define, for $h > 0$,
the following DOF formula
}
\begin{align}\label{eq:edof}
  y \mapsto & ~ \EDOF{\HT}(y, \lambda, h)
  =
  \#\{ |y| > \lambda \} \;+
  \longversion{}{
  \nonumber\\
  &
  }%
  \; \tfrac{\lambda \sqrt{\sigma^2\!+\!h^2}}{\sqrt{2 \pi} \sigma h}
  \sum_{i = 1}^P
  \left[
    \exp\pa{\!-\tfrac{(y_i\!+\!\lambda)^2}{2 h^2}}
    \!+\!
    \exp\pa{\!-\tfrac{(y_i\!-\!\lambda)^2}{2 h^2}}
    \right]\longversion{.}{}
\end{align}
\longversion{
  We now give our two main results
  proved in Section \ref{sec:proof}.
}{
where $\#\{ |y| \!>\! \lambda \}$ is the number of entries of $|y|$ greater than $\lambda$.
}%
\begin{thm}\label{thm:score}
  Let $Y = x_0 + W$ for $W \sim \Nn(x_0, \sigma^2\Id_P)$. Take $\widehat{h}(P)$ such that $\lim_{P \to \infty} \widehat{h}(P) = 0$ and
  $\lim_{P \to \infty} P^{-1} \widehat{h}(P)^{-1} = 0$. Then $\plim_{P \to \infty}\tfrac{1}{P}\pa{\EDOF{\HT}(Y, \lambda, \widehat{h}(P))-\DOF{\HT}(x_0,\lambda)}=0$. In particular
  \begin{align*}
    1. \lim_{P \to \infty} \EE_W \big[ \tfrac{1}{P} \EDOF{\HT}(Y, \lambda, \widehat{h}(P)) \big] &=
     \lim_{P \to \infty} \tfrac{1}{P} \DOF{\HT}(x_0,\lambda), ~ \text{and} ~ \\
    2. \lim_{P \to \infty} \VV_W \big[ \tfrac{1}{P} \EDOF{\HT}(Y, \lambda, \widehat{h}(P)) \big]  &= 0 ~,
  \end{align*}
  where $\VV_W$ is the variance w.r.t. $W$.
\end{thm}
\longversion{
  We now turn to a straightforward corollary of this theorem.
}{
  The proof is available in the extended version of this abstract
  \cite{deledalle2013score}.
}%
\longversion{
\begin{cor}
  Let $Y = x_0 + W$ for $W \sim \Nn(x_0, \sigma^2\Id_P)$, and assume that $\norm{x_0}_4=o(P^{1/2})$. Take $\widehat{h}(P)$ such that $\lim_{P \to \infty} \widehat{h}(P) = 0$ and
  $\lim_{P \to \infty} P^{-1} \widehat{h}(P)^{-1} = 0$. Then, the Stein COnsistent Risk Estimator (SCORE) evaluated at a realization $y$ of $Y$
  \begin{align*}
    \mathrm{SCORE}\{x\}(y, \lambda, \widehat{h}(P))
    =
    \sum_{i = 1}^P \left(
    (y_i^2 - \sigma^2) + I(|y_i| > \lambda) (2\sigma^2 - y_i^2)
    \longversion{}{\\}%
    + 2 \sigma
    \tfrac{\lambda \sqrt{\sigma^2\!+\!\widehat{h}(P)^2}}{\sqrt{2 \pi} \widehat{h}(P)}
    \left[
      \exp\pa{\!-\tfrac{(y_i\!+\!\lambda)^2}{2 \widehat{h}(P)^2}}
      \!+\!
      \exp\pa{\!-\tfrac{(y_i\!-\!\lambda)^2}{2 \widehat{h}(P)^2}}
      \right]
    \right)
  \end{align*}
  is such that $\plim_{P \to \infty}\frac{1}{P}\pa{\mathrm{SCORE}\{x\}(Y, \lambda, \widehat{h}(P))-\EE_W \norm{\HT(Y, \lambda) - x_0}^2}=0$.
\end{cor}
}{
An immediate corollary of Theorem \ref{thm:score},
also given in \cite{deledalle2013score}, is that
\eqref{eq:edof} and~\eqref{eq:erisk} provide together
the Stein COnsistent Risk Estimator (SCORE) which is biased but consistent.
}%
Fig.~\ref{fig:algo} summarizes the pseudo-code when applying SCORE to automatically find
the optimal threshold $\lambda$ that minimizes SCORE in a predefined (non-empty) range.

\begin{figure}
\centering
\begin{minipage}{\linewidth}
\rule{\linewidth}{2px}
\begin{algorithmic}[0]
  \vspace{-0.1cm}
  \State \textbf{Algorithm} Risk estimation for Hard Thresholding
\end{algorithmic}
\vspace{-0.2cm}
\rule{\linewidth}{1px}
\begin{algorithmic}[0]
  \vspace{-0.1cm}
  \State \makebox[3cm][l]{\textbf{Inputs:}} observation $y \in \RR^P$, threshold $\lambda > 0$
  \State \makebox[3cm][l]{\textbf{Parameters:}} noise variance $\sigma^2 > 0$
  \State \makebox[3cm][l]{\textbf{Output:}} solution $x^\star$\\\vspace{-0.3cm}
  \State Initialize $h \leftarrow \widehat{h}(P)$
  \ForAll {$\lambda$ in the tested range}
  \State Compute $x \leftarrow \HT(y,\lambda)$ using \eqref{eq:ht}
  \State Compute $\EDOF{\HT}(y, \lambda, h)$ using \eqref{eq:edof}
  \State Compute $\mathrm{SCORE}$ at $y$ using \eqref{eq:erisk}
  \EndFor
  \State \Return $x^\star \leftarrow x$ that provides the smallest $\mathrm{SCORE}$
\end{algorithmic}
\vspace{-0.3cm}
\rule{\linewidth}{1px}
\end{minipage}
\caption{Pseudo-algorithm for HT
  with SCORE-based threshold optimization.}
  \label{fig:algo}
\end{figure}
\begin{figure}
  \centering
  \longversion{
    \includegraphics[width=0.5\linewidth]{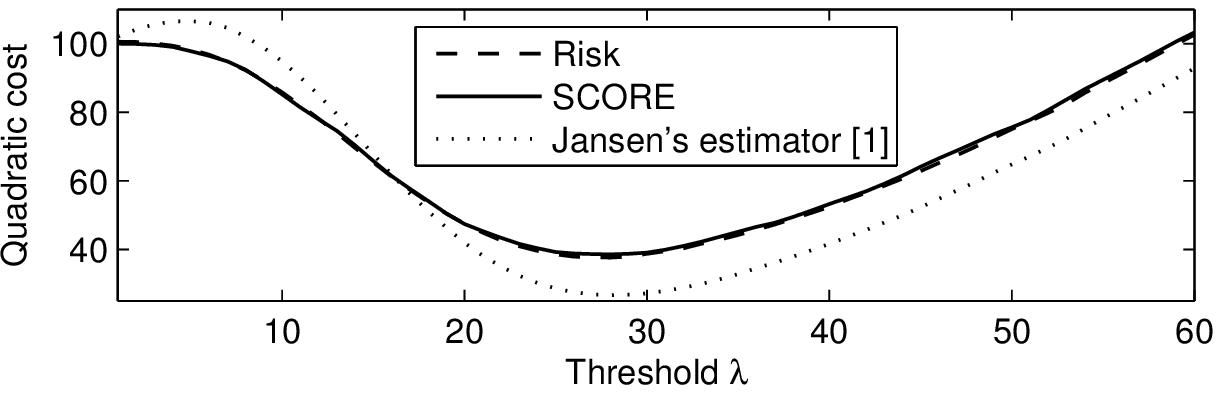}%
  }{
    \includegraphics[width=1\linewidth]{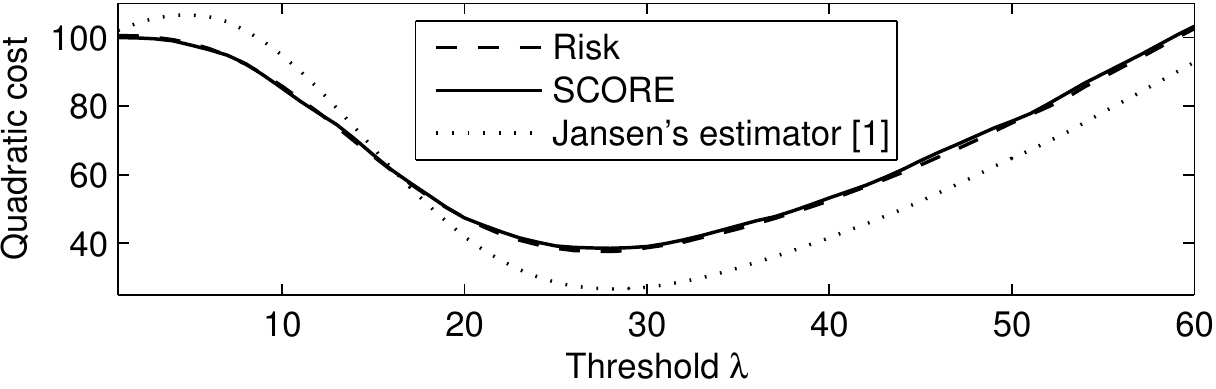}%
  }
  \longversion{}{\vspace{-0.6cm}}
  \caption{Risk and its SCORE estimate with respect to the threshold $\lambda$.}
  \label{fig:score}
  \vspace{-0.6cm}
\end{figure}

\section{Experiments and conclusions}

Fig.~\ref{fig:score} shows the evolution of the true
risk, the SCORE and the risk estimator of \cite{jansen2011icv} as a function of $\lambda$
where $x_0$ is a compressible vector of length $P = 2\text{E}5$ whose sorted values in magnitude decay as $|x_0|_{(i)} = 1/i^\gamma$ for $\gamma > 0$,
and we have chosen $\sigma$ such that the
SNR of $y$ is of about $5.65$dB and
$\widehat{h}(P) = 6\sigma/P^{1/3} \approx \sigma/10$.
The optimal $\lambda$ is found around the minimum of the true risk.

Future work will concern a deeper investigation of the choice of $\widehat{h}(P)$,
comparison with other biased risk estimators,
and extensions to other non-continuous estimators and inverse problems.

\longversion{
\section{Proof}\label{sec:proof}

We first derive a closed-form expression for the DOF of HT.

\begin{lem}\label{lem:HTdof}
  Let $Y = x_0 + W$ where $W \sim \Nn(x_0, \sigma^2\Id_P)$.
  The DOF of HT is given by
  \begin{align}
    \label{eq:dof}
    &\DOF{\HT}(x_0,\lambda) =
    \left(P -
    \frac{1}{2}
    \sum_{i = 1}^P
    \left[
      \erf \pa{\frac{(x_0)_i + \lambda}{\sqrt{2} \sigma}}
      -
      \erf \pa{\frac{(x_0)_i - \lambda}{\sqrt{2} \sigma}}
      \right]
    \right)
    +
    \frac{\lambda}{\sqrt{2 \pi} \sigma}
    \sum_{i = 1}^P
    \left[
      \exp \pa{-\frac{((x_0)_i \!+\! \lambda)^2}{2 \sigma^2}}
      \!+\!
      \exp \pa{-\frac{((x_0)_i \!-\! \lambda)^2}{2 \sigma^2}}
      \right].
  \end{align}
\end{lem}
\begin{proof}
  According to \cite{jansen2011icv}, we have
  \begin{align*}
    \DOF{\HT}(x_0,\lambda) = \EE_W [ \#\{ |Y| > \lambda \} ] + \lambda/\sigma^2 \EE_W \left[ \sum_{i=1}^P \sign(Y_i) W_i I(|Y_i| \!>\! \lambda) \right]
  \end{align*}
  where
  $\sign(.)$ is the sign function and $I(\omega)$ is the indicator for an event $\omega$.
  Integrating w.r.t. to the zero-mean Gaussian density of variance $\sigma^2$ yields the closed form of the expectation terms.
\end{proof}

We now turn to the proof of our theorem.
\begin{proof}
The first part of \eqref{eq:dof} corresponds to
$\EE_W[\#\{ |Y| > \lambda \}]$, and can then be obviously unbiasedly
estimated from an observation $y$ by $\#\{ |y| > \lambda \}$.
Let $A$ be the function defined, for $(t,a) \in \RR^2$, by
\begin{align*}
  A(t, a) = \frac{\sqrt{\sigma^2 + h^2}}{h} \exp\pa{-\frac{(t-a)^2}{2 h^2}} ~.
\end{align*}
By classical convolution properties of Gaussians, we have
\begin{align*}
\EE_{W_i} [ A(Y_i, a) ] &= \exp\pa{-\frac{((x_0)_i-a)^2}{2 (\sigma^2 + h^2)}}, \qandq\\
\VV_{W_i} [ A(Y_i, a) ] &=
  \frac{\sigma^2 \!+\! h^2}{h \sqrt{2 \sigma^2 \!+\!h^2}}
  \exp \pa{\!\!-\frac{((x_0)_i\!-\!a)^2}{2\sigma^2+h^2}}
  \!-\! \exp\pa{\!\!-\frac{((x_0)_i\!-\!a)^2}{\sigma^2 \!+\!h^2}}.
\end{align*}
Taking $h = \widehat{h}(P)$ and assuming
$\lim_{P\to \infty} \widehat{h}(P) = 0$
shows that
\begin{align*}
&\lim_{P \to \infty}
\EE_W \left[ \frac{1}{P} \sum_{i=1}^P A(Y_i, a) \right] =
\lim_{P\to \infty} \frac{1}{P} \sum_{i=1}^P \EE_W \left[ A(Y_i, a) \right] =
\lim_{P\to \infty} \frac{1}{P} \sum_{i=1}^P \exp\pa{-\frac{((x_0)_i\!-\!a)^2}{2 \sigma^2}} ~.
\end{align*}
Since from \eqref{eq:edof}, we have
\[
\EDOF{\HT}(Y, \lambda, h) = \#\{ |Y| > \lambda \} + \frac{\lambda}{\sqrt{2 \pi} \sigma}
  \sum_{i = 1}^P
  \left[A(Y_i,\lambda) + A(Y_i,-\lambda)\right]
\]
and using Lemma~\ref{lem:HTdof}, statement 1. follows.

For statement 2., the Cauchy-Schwartz inequality implies that
\[
\VV_W \left[ \frac{1}{P} \EDOF{\HT}(Y, \lambda, h)\right]^{1/2} \leq \frac{\VV_W \left[\#\{ |Y| > \lambda \}\right]^{1/2}}{P} + \frac{\lambda^2}{2 \pi P} \sum_{i = 1}^P
  \left[\VV_W \left[A(Y_i,\lambda)\right]^{1/2} + \VV_W \left[A(Y_i,-\lambda)\right]^{1/2}\right] ~.
\]
$\#\{ |Y_i| > \lambda \} \sim_{\mathrm{iid}} \mathrm{Bin}(P,1-p)$ whose variance is $Pp(1-p)$, where $p=\frac{1}{2}\pa{\erf \pa{\frac{(x_0)_i + \lambda}{\sqrt{2} \sigma}} -\erf \pa{\frac{(x_0)_i - \lambda}{\sqrt{2} \sigma}}}$. It follows that 
\[
\lim_{P \to \infty }\VV_W \left[\frac{1}{P}\#\{ |Y| > \lambda \}\right] = 0 ~.
\]
Taking again $h = \widehat{h}(P)$ with
$\lim_{P\to \infty} \widehat{h}(P) = 0$
and $\lim_{P\to \infty} P^{-1}\widehat{h}(P)^{-1} = 0$, yields
\begin{align*}
&\lim_{P \to \infty}
\VV_W \left[ \frac{1}{P} \sum_{i=1}^P A(Y_i, a) \right] =
\lim_{P\to \infty} \frac{1}{P^2} \sum_{i=1}^P \VV_W \left[ A(Y_i, a) \right] =
0 ~,
\end{align*}
where we used the fact that the random variables $Y_i$ are uncorrelated. This establishes 2.. Consistency (i.e. convergence in probability) follows from traditional arguments by invoking Chebyshev inequality and using asymptotic unbiasedness and vanishing variance established in 1. and 2..
\end{proof}

Let us now prove the corollary.
\begin{proof}
By assumption, $\lim_{P\to \infty} \widehat{h}(P) = 0$. Thus by by virtue of statement 1. of Theorem~\ref{thm:score} and specializing \eqref{eq:erisk} to the case of HT gives
\begin{align*}
  \lim_{P \to \infty}
  \EE_W \left[ \frac{1}{P} \mathrm{SCORE}\{\HT\}(y, \lambda, \widehat{h}(P)) \right]
  &= \lim_{P \to \infty}
  \frac{1}{P} \EE_W \left[ \norm{Y - \HT(Y, \lambda))}^2
  \!-\! P \sigma^2
  \!+\! 2 \sigma^2 \EDOF{\HT}(Y, \lambda, \widehat{h}(P)) \right] \\
  &= \lim_{P \to \infty}
  \frac{1}{P} \EE_W \norm{Y - \HT(Y, \lambda))}^2
  \!-\!  \sigma^2
  \!+\! 2 \sigma^2 \lim_{P \to \infty}\frac{1}{P} \EE_W \EDOF{\HT}(Y, \lambda, \widehat{h}(P)) \\
  & = \lim_{P \to \infty}
  \frac{1}{P} \EE_W \norm{Y - \HT(Y, \lambda))}^2
  \!-\!  \sigma^2
  \!+\! 2 \sigma^2 \lim_{P \to \infty}\frac{1}{P} \DOF{\HT}(x_0,\lambda) \\
  & = \lim_{P\to \infty}
  \frac{1}{P} \EE_W \| \HT(y, \lambda) - x_0 \|^2
\end{align*}
where we used the fact that all the limits of the expectations are finite.
The Cauchy-Schwartz inequality again yields
\begin{align*}
  \VV_W \left[ \frac{1}{P} \mathrm{SCORE}\{\HT\}(Y, \lambda, \widehat{h}(P)) \right]^{1/2}
  &\leq
    \VV_W \left[ \frac{1}{P} \pa{\| Y - \HT(Y, \lambda) \|^2} \right]^{1/2}
+
    2\sigma^2\VV_W \left[ \frac{1}{P} \EDOF{\HT}(Y, \lambda, \widehat{h}(P) ) \right]^{1/2} \\
  &=\frac{1}{P}\pa{\sum_{i=1}^P \VV_{W_i} \left[|Y_i|^2 I(|Y_i| < \lambda) \right]}^{1/2}
+
    2\sigma^2\VV_W \left[ \frac{1}{P} \EDOF{\HT}(Y, \lambda, \widehat{h}(P) ) \right]^{1/2} \\
  &\leq\frac{1}{P}\pa{\sum_{i=1}^P \EE_{W_i} |Y_i|^4 }^{1/2}
+
    2\sigma^2\VV_W \left[ \frac{1}{P} \EDOF{\HT}(Y, \lambda, \widehat{h}(P) ) \right]^{1/2} \\
  &= \pa{\frac{\norm{x_0}_4^4}{P^2} + 6\sigma^2\frac{\norm{x_0}^2}{P^2} + \frac{3\sigma^4}{P}}^{1/2}
+
    2\sigma^2\VV_W \left[ \frac{1}{P} \EDOF{\HT}(Y, \lambda, \widehat{h}(P) ) \right]^{1/2} \\
  &\leq  \pa{\pa{\frac{\norm{x_0}_4^2}{P}}^2 + 6\sigma^2\frac{\norm{x_0}_4^2}{P} + \frac{3\sigma^4}{P}}^{1/2}
+
    2\sigma^2\VV_W \left[ \frac{1}{P} \EDOF{\HT}(Y, \lambda, \widehat{h}(P) ) \right]^{1/2} ~.
\end{align*}
As by assumption, $\lim_{P\to \infty} P^{-1} \widehat{h}(P)^{-1} = 0$ and $\norm{x_0}_4=o\pa{P^{1/2}}$, the variance of SCORE vanishes as $P \to \infty$. We conclude using the same convergence in probability arguments used at the end of the proof of Theorem~\ref{thm:score}.  
\end{proof}
}{}

\bibliographystyle{IEEEtran}

\end{document}